\documentclass[12pt]{article}
\usepackage{geometry,amsthm,amssymb,amsmath,enumerate,float,cite,algorithm2e,verbatim,tikz}
\usetikzlibrary{decorations.pathreplacing}
\newtheorem{theorem}{Theorem}
\newtheorem{conjecture}{Conjecture}
\newtheorem{lemma}{Lemma}
\newtheorem{claim}{Claim}

\title{Bounding and approximating minimum maximal
matchings in regular graphs}
\author{
Julien Baste$^1$ \and
Maximilian F\"{u}rst$^1$ \and
Michael A. Henning$^2$ \and
Elena Mohr$^1$ \and
Dieter Rautenbach$^1$}
\date{}

\begin{document}

\maketitle

{\small
\begin{center}
$^1$
Institute of Optimization and Operations Research, Ulm University, Germany,
\texttt{\{julien.baste,maximilian.fuerst,elena.mohr,dieter.rautenbach\}@uni-ulm.de}\\[3mm]
$^2$
Department of Mathematics and Applied Mathematics, University of Johannesburg,
Auckland Park, 2006, South Africa,
\texttt{mahenning@uj.ac.za}
\end{center}
}

\begin{abstract}
The edge domination number $\gamma_e(G)$ of a graph $G$
is the minimum size of a maximal matching in $G$.
It is well known that this parameter is computationally very hard,
and several approximation algorithms and heuristics have been studied.
In the present paper,
we provide best possible upper bounds on $\gamma_e(G)$
for regular and non-regular graphs $G$
in terms of their order and maximum degree.
Furthermore, we discuss algorithmic consequences
of our results and their constructive proofs.
\end{abstract}
{\small
\begin{tabular}{lp{13cm}}
{\bf Keywords:} & Minimum maximal matching; edge domination\\
{\bf MSC 2010:} & 05C69, 05C70
\end{tabular}
}

\pagebreak

\section{Introduction}

We consider finite, simple, and undirected graphs, and use standard terminology.
Let $G$ be a graph and let $M$ be a set of edges of $G$.
Let $V(M)$ denote the set of vertices of $G$ that are incident with an edge in $M$.
The set $M$ is a {\it matching} in $G$ if the edges in $M$ are pairwise disjoint.
A matching $M$ in $G$ is {\it maximal} if it is maximal with respect to inclusion,
that is, the set $V(G)\setminus V(M)$ is independent.
Let the {\it edge domination number} $\gamma_e(G)$ of $G$
be the minimum size of a maximal matching in $G$.
A maximal matching in $G$ of size $\gamma_e(G)$ is a {\it minimum maximal matching}.

The edge domination number and minimum maximal matchings have been studied for a long time.
Yannakakis and Gavril \cite{yaga} showed that finding a minimum maximal matching
is NP-hard even for planar graphs or bipartite graphs of maximum degree $3$.
Stronger hardness and inapproximability results were obtained \cite{deek,chch,hoki},
and heuristics as well as approximation algorithms were studied  \cite{caek,calalaleme,calale,cafukopa,funa,golera}.

In the present paper we consider upper bounds on the edge domination number
and their algorithmic consequences.
We state the following conjecture as a starting point.

\begin{conjecture}\label{conjecture1}
If $G$ is a connected $\Delta$-regular graph of order $n$ for some $\Delta\geq 3$, then
\begin{eqnarray}\label{econj}
\gamma_e(G) & \leq & \frac{2\Delta-1}{4\Delta}n+\frac{1}{2}
\end{eqnarray}
with equality in (\ref{econj}) if and only if
$G$ has a spanning subgraph that is the union of an odd number of copies of $K_{\Delta,\Delta}-e$,
see Figure \ref{fig1}.
\end{conjecture}

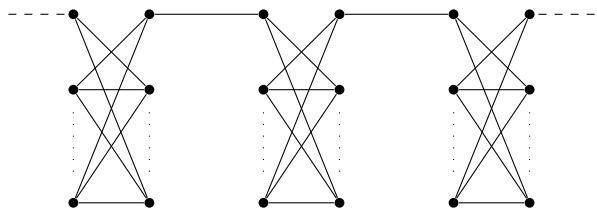
\begin{figure}[H]
\begin{center}
\begin{tikzpicture}

\node[fill, circle, inner sep=1.3pt] (v1) at (0,0) {};
\node[fill, circle, inner sep=1.3pt] (v2) at (0,-1) {};
\node[fill, circle, inner sep=1.3pt] (v3) at (0,-2.5) {};

\node[fill, circle, inner sep=1.3pt] (v4) at (1,0) {};
\node[fill, circle, inner sep=1.3pt] (v5) at (1,-1) {};
\node[fill, circle, inner sep=1.3pt] (v6) at (1,-2.5) {};

\draw (v1) -- (v5);
\draw (v1) -- (v6);
\draw (v2) -- (v4);
\draw (v2) -- (v5);
\draw (v2) -- (v6);
\draw (v3) -- (v4);
\draw (v3) -- (v5);
\draw (v3) -- (v6);
\draw [loosely dotted ] (0,-1.3)--(0,-2.2);
\draw [loosely dotted ] (1,-1.3)--(1,-2.2);

\node[fill, circle, inner sep=1.3pt] (u1) at (2.5,0) {};
\node[fill, circle, inner sep=1.3pt] (u2) at (2.5,-1) {};
\node[fill, circle, inner sep=1.3pt] (u3) at (2.5,-2.5) {};

\node[fill, circle, inner sep=1.3pt] (u4) at (3.5,0) {};
\node[fill, circle, inner sep=1.3pt] (u5) at (3.5,-1) {};
\node[fill, circle, inner sep=1.3pt] (u6) at (3.5,-2.5) {};

\draw (u1) -- (u5);
\draw (u1) -- (u6);
\draw (u2) -- (u4);
\draw (u2) -- (u5);
\draw (u2) -- (u6);
\draw (u3) -- (u4);
\draw (u3) -- (u5);
\draw (u3) -- (u6);
\draw [loosely dotted ] (2.5,-1.3)--(2.5,-2.2);
\draw [loosely dotted ] (3.5,-1.3)--(3.5,-2.2);

\node[fill, circle, inner sep=1.3pt] (w1) at (5,0) {};
\node[fill, circle, inner sep=1.3pt] (w2) at (5,-1) {};
\node[fill, circle, inner sep=1.3pt] (w3) at (5,-2.5) {};

\node[fill, circle, inner sep=1.3pt] (w4) at (6,0) {};
\node[fill, circle, inner sep=1.3pt] (w5) at (6,-1) {};
\node[fill, circle, inner sep=1.3pt] (w6) at (6,-2.5) {};

\draw (w1) -- (w5);
\draw (w1) -- (w6);
\draw (w2) -- (w4);
\draw (w2) -- (w5);
\draw (w2) -- (w6);
\draw (w3) -- (w4);
\draw (w3) -- (w5);
\draw (w3) -- (w6);
\draw [loosely dotted ] (5,-1.3)--(5,-2.2);
\draw [loosely dotted ] (6,-1.3)--(6,-2.2);
\node(h0) at (-1,0) {};
\node(h1) at (7,0) {};
\draw (v4) -- (u1);
\draw (u4) -- (w1);
\draw[dashed] (h0)--(v1);
\draw[dashed] (h1)--(w4);

\end{tikzpicture}
\end{center}
\caption{The extremal graphs of Conjecture \ref{conjecture1}.}\label{fig1}
\end{figure}
Our two main results imply weak versions of Conjecture \ref{conjecture1}.
Firstly, we prove Conjecture \ref{conjecture1} for $\Delta=3$ under more restrictive assumptions.
Let $T^*$ be the tree that arises by subdividing exactly two edges of a claw $K_{1,3}$ exactly once.
Recall that a graph is $T^*$-free if it does not contain $T^*$ as an induced subgraph.

\begin{theorem}\label{theorem3}
If $G$ is a connected cubic graph of order $n$
that is bipartite and $T^*$-free, then
\begin{eqnarray*}
\gamma_e(G) & \leq & \frac{5}{12}n+\frac{1}{2}.
\end{eqnarray*}
\end{theorem}
In view of the extremal graphs from Conjecture \ref{conjecture1},
the bound in Theorem \ref{theorem3} is still best possible.
Secondly, we prove a weaker bound for general $\Delta$.
\begin{theorem}\label{theorem4}
If $G$ is a connected $\Delta$-regular graph of order $n$ for some $\Delta\geq 3$, then
\begin{eqnarray}\label{e1}
\gamma_e(G)\leq \frac{\Delta(2\Delta-3)n+2\Delta}{2(\Delta-1)(2\Delta-1)}
\end{eqnarray}
with equality if and only if $G$ is $K_{\Delta,\Delta}$.
\end{theorem}
The proofs of Theorems \ref{theorem3} and \ref{theorem4}
rely on bounds for non-regular graphs given in the next section
that are of interest on their own right.

Since our proofs yield efficient algorithms to construct maximal matchings
whose sizes are at most the corresponding upper bounds,
some simple lower bounds on the edge domination number
based on double counting yield approximation algorithms
that improve some known results \cite{caek,calalaleme,calale,cafukopa,funa,golera}
for restricted classes of graphs.

\begin{theorem}\label{theorem5}
Let $\Delta$ be an integer at least $3$.

For every fixed $\epsilon>0$,
there are polynomial time approximation algorithms
for the minimum maximal matching problem
that have approximation ratios
\begin{itemize}
\item $\frac{25}{18}+\epsilon$ for connected cubic bipartite $T^*$-free graphs,
\item $2-\frac{1}{\Delta-1}+\epsilon$ for connected $\Delta$-regular graphs, and
\item $1+\frac{4\Delta-7}{2\Delta^2-3\Delta+1}+\epsilon$ for connected claw-free $\Delta$-regular graphs.
\end{itemize}
\end{theorem}
All proofs are given in the following section.

\section{Proofs}

The proof of Theorem \ref{theorem3} is based on the following result.

\begin{theorem}\label{theorem1}
If $G$ is a subcubic bipartite $T^*$-free graph of order $n$ and size $m$
such that no component of $G$ is cubic, then
\begin{eqnarray}\label{e10}
\gamma_e(G)\leq \frac{2n}{3}-\frac{m}{6}.
\end{eqnarray}
Furthermore,
a maximal matching whose size is at most the right hand side of (\ref{e10}) can be found efficiently.
\end{theorem}
\begin{proof}
We first prove the statement about the bound,
and then argue that our proof yields an efficient algorithm to determine the desired maximal matching.
Suppose, for a contradiction, that $G$ is a counterexample of minimum order.
Clearly, this implies that $G$ is connected and has order at least $3$.

At several occasions, we shall consider subgraphs $G',G'',\ldots$, of $G$,
whose orders and sizes we always denote by $n',n'',\ldots$ and $m',m'',\ldots$, respectively.

\begin{claim}\label{claim1}
The minimum degree of $G$ is $2$.
\end{claim}
\begin{proof}[Proof of Claim \ref{claim1}]
Suppose, for a contradiction, that $u$ is a vertex of degree $1$ in $G$.
Let $v$ be the unique neighbor of $u$, and let $w$ be a neighbor of $v$ distinct from $u$.
If $G'=G-\{ u,v,w\}$, then $n'=n-3$ and $m'\geq m-5$.
Since adding $vw$ to a maximal matching in $G'$ yields a maximal matching in $G$,
the choice of $G$ implies
\begin{eqnarray*}
\gamma_e(G)
& \leq & \frac{2}{3}n'-\frac{1}{6}m'+1
\leq \frac{2}{3}(n-3)-\frac{1}{6}(m-5)+1
< \frac{2}{3}n-\frac{1}{6}m.
\end{eqnarray*}
This contradiction completes the proof of the claim.
\end{proof}

\begin{claim}\label{claim2}
No cycle of length $4$ in $G$ contains a vertex of degree $2$.
\end{claim}
\begin{proof}[Proof of Claim \ref{claim2}]
Suppose, for a contradiction, that $uv_1wv_2u$ is a cycle in $G$ such that $u$ has degree $2$.
Since $C_4$ is not a counterexample,
the graph $G$ is not $C_4$, and we may assume that $v_1$ has degree $3$.
Let $w_1$ be the neighbor of $v_1$ distinct from $u$ and $w$.
Let $G'=G-\{ u,v_1,v_2,w,w_1\}$.
Note that
$m'\geq m-9$, and that
adding $\{ v_1w_1, v_2w\}$ to a maximal matching in $G'$ yields a maximal matching in $G$.
If $m'\geq m-8$, then we obtain a similar contradiction as in the proof of Claim \ref{claim1}.
Hence, $m'=m-9$, which implies that $v_2$, $w$, and $w_1$ also have degree $3$,
and that $u$ and $w$ are the only common neighbors of $v_1$ and $v_2$.
Let $w_2$ be the neighbor of $v_2$ distinct from $u$ and $w$,
and let $z$ be the neighbor of $w$ distinct from $v_1$ and $v_2$.
Since $G$ is bipartite, the vertex $z$ is distinct from $w_1$ and $w_2$.
By symmetry, the degree of $w_2$ is also $3$.
If $z$ is not adjacent to $w_1$, and $z'$ is a neighbor of $w_1$ distinct from $v_1$,
then, as illustrated in Figure \ref{fig2.1}, $G[\{ u,v_1,w,w_1,z,z'\}]$ equals $T^*$, which is a contradiction.
Hence, by symmetry, the vertex $z$ is adjacent to $w_1$ and $w_2$.
Note that $w_1$ has a neighbor $z''$ distinct from $v_1$ and $z$.
If $z''$ is adjacent to $w_2$, then, since $G$ is $T^*$-free,
the vertex $z''$ has degree $2$, $n=8$, and $G$ is no counterexample.
Hence, the vertex $z''$ is not adjacent to $w_2$,
and $G[\{ u,v_1,w_1,w_2,z,z''\}]$ equals $T^*$ as illustrated in Figure \ref{fig2.2},
which is a contradiction, and completes the proof of the claim.
\end{proof}

\noindent
\begin{minipage}[t]{0.5\textwidth}
\begin{figure}[H]
\begin{center}
\begin{tikzpicture}

\node[fill, circle, inner sep=1.3pt, label=right:$u$] (u) at (0,3) {};
\node[fill, circle, inner sep=1.3pt, label=left:$v_1$] (v1) at (-1,2) {};
\node[fill, circle, inner sep=1.3pt, label=right:$v_2$] (v2) at (1,2) {};
\node[fill, circle, inner sep=1.3pt, label=right:$w$] (w) at (0,1) {};

\node[fill, circle, inner sep=1.3pt, label=left:$w_1$] (w1) at (-1.5,1) {};
\node[fill, circle, inner sep=1.3pt, label=right:$w_2$] (w2) at (1.5,1) {};
\node[fill, circle, inner sep=1.3pt, label=below:$z$] (z) at (0,0) {};
\node[fill, circle, inner sep=1.3pt, label=left:$z'$] (z1) at (-2,0) {};

\node(h1) at (-1,0) {};
\node(h2) at (1,0) {};
\node(h3) at (2,0) {};

\draw[line width=0.75mm] (u)--(v1);
\draw (u)--(v2);
\draw[line width=0.75mm] (w1)--(v1);
\draw (w2)--(v2);
\draw[line width=0.75mm] (w)--(v1);
\draw (w)--(v2);
\draw[line width=0.75mm] (w)--(z);
\draw[line width=0.75mm] (w1)--(z1);

\draw[dashed] (w1)--(h1);
\draw[dashed] (w2)--(h2);
\draw[dashed] (w2)--(h3);

\end{tikzpicture}
\end{center}
\caption{$T^*$ in the first case of Claim 2.}\label{fig2.1}
\end{figure}

\end{minipage}\begin{minipage}[t]{0.5\textwidth}
\begin{figure}[H]
\begin{center}
\begin{tikzpicture}
\node[fill, circle, inner sep=1.3pt, label=right:$u$] (u) at (0,3) {};
\node[fill, circle, inner sep=1.3pt, label=left:$v_1$] (v1) at (-1,2) {};
\node[fill, circle, inner sep=1.3pt, label=right:$v_2$] (v2) at (1,2) {};
\node[fill, circle, inner sep=1.3pt, label=right:$w$] (w) at (0,1) {};

\node[fill, circle, inner sep=1.3pt, label=left:$w_1$] (w1) at (-1.5,1) {};
\node[fill, circle, inner sep=1.3pt, label=right:$w_2$] (w2) at (1.5,1) {};
\node[fill, circle, inner sep=1.3pt, label=below:$z$] (z) at (0,0) {};
\node[fill, circle, inner sep=1.3pt, label=left:$z''$] (z1) at (-2,0) {};

\node(h1) at (-1,0) {};
\node(h2) at (1,0) {};
\node(h3) at (2,0) {};

\draw[line width=0.75mm] (u)--(v1);
\draw (u)--(v2);
\draw[line width=0.75mm] (w1)--(v1);
\draw (w2)--(v2);
\draw (w)--(v1);
\draw (w)--(v2);
\draw (w)--(z);
\draw[line width=0.75mm] (w1)--(z1);

\draw[line width=0.75mm] (w1)--(z);
\draw[line width=0.75mm] (w2)--(z);
\draw[dashed] (w2)--(h3);
\end{tikzpicture}
\end{center}
\caption{$T^*$ in the second case of Claim 2.}\label{fig2.2}
\end{figure}
\end{minipage}
\bigskip

Let $u$ be a vertex of degree $2$ with neighbors $v_1$ and $v_2$.
By Claim \ref{claim2},
the vertices $v_1$ and $v_2$ have no common neighbor except for $u$.
Let $w_i$ be a neighbor of $v_i$ distinct from $u$ for $i\in [2]$.
Let $G'=G-\{ u,v_1,v_2,w_1,w_2\}$.
Note that
$m'\geq m-10$, and that
adding $\{ v_1w_1, v_2w_2\}$ to a maximal matching in $G'$ yields a maximal matching in $G$.
If $m'\geq m-8$, then we obtain a similar contradiction as in the proof of Claim \ref{claim1}.
Hence, we may assume that $m'\leq m-9$.

First suppose that $v_2$ has degree $2$.
Since $m'\leq m-9$,
the vertex $v_1$ has a neighbor $w_1'$ distinct from $u$ and $w_1$,
and $w_1$ has two neighbors $z_1$ and $z_1'$ distinct from $v_1$.
Since $G$ is $T^*$-free, the vertex $w_1'$ is adjacent to $z_1$ and $z_1'$.
By Claim \ref{claim2}, the vertex $z_1$ has a neighbor $x$ distinct from $w_1$ and $w_1'$.
Since $G$ is $T^*$-free, the vertex $z_1'$ is adjacent to $x$.
If $x$ equals $w_2$, the graph is as in Figure \ref{fig4}, which is a contradiction.
\begin{figure}[H]
\begin{center}
\begin{tikzpicture}
\node[fill, circle, inner sep=1.3pt, label=right:$u$] (u) at (0,3) {};
\node[fill, circle, inner sep=1.3pt, label=left:$v_1$] (v1) at (-1,2) {};
\node[fill, circle, inner sep=1.3pt, label=right:$v_2$] (v2) at (1,2) {};

\node[fill, circle, inner sep=1.3pt, label=above:$w_1$] (w1) at (-0.5,1.5) {};
\node[fill, circle, inner sep=1.3pt, label=below:$w_1'$] (w2) at (-0.5,0.7) {};
\node[fill, circle, inner sep=1.3pt, label=right:$w_2$] (w3) at (1,1.2) {};
\node[fill, circle, inner sep=1.3pt, label=above:$z_1$] (z1) at (0.5,1.5) {};
\node[fill, circle, inner sep=1.3pt, label=below:$z_1'$] (z2) at (0.5,0.7) {};

\node(h1) at (-1,0) {};
\node(h2) at (1,0) {};
\node(h3) at (2,0) {};

\draw(u)--(v1);
\draw (u)--(v2);
\draw(w1)--(v1);
\draw (w2)--(v1);
\draw(w1)--(z1);
\draw(w1)--(z2);
\draw(w2)--(z1);
\draw(w2)--(z2);
\draw(z1)--(w3);
\draw(z2)--(w3);
\draw(v2)--(w3);

\end{tikzpicture}
\end{center}
\caption{The graph $G$ in the case that $x=w_2$.
Note that $G$ has a maximal matching missing $v_1$ and $w_2$,
which implies $\gamma_e(G)<\frac{2n}{3}-\frac{m}{6}$.}\label{fig4}
\end{figure}
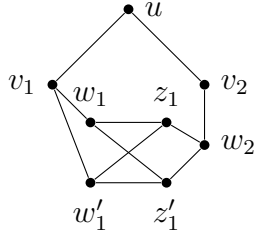

Hence, the vertices $x$ and $w_2$ are distinct.
By Claim \ref{claim2}, the vertex $x$ has a neighbor $y$ distinct from $z_1$ and $z_1'$.
Let $G''=G-\{ u,v_1,v_2,w_1,w_1',w_2,z_1,z_1',x,y\}$.
Since
$n'=n-10$,
$m'\geq m-16$, and
adding $\{ v_1w_1',v_2w_2,w_1z_1',xy\}$ to a maximal matching in $G''$
yields a maximal matching in $G$,
the choice of $G$ implies
\begin{eqnarray*}
\gamma_e(G)
& \leq & \frac{2}{3}n''-\frac{1}{6}m''+4
\leq \frac{2}{3}(n-10)-\frac{1}{6}(m-16)+4
=\frac{2}{3}n-\frac{1}{6}m.
\end{eqnarray*}
This contradiction implies that we may assume that $v_1$ and $v_2$ both have degree $3$.

Let $N_G(v_1)=\{ u,w_i,w_i'\}$ for $i\in [2]$.
By Claim \ref{claim2}, the vertices $w_1$, $w_1'$, $w_2$, and $w_2'$ are all distinct.
Since $m'\leq m-9$, we may assume, by symmetry,
that $w_1$ has two neighbors $z_1$ and $z_1'$ distinct from $v_1$.
Since $G$ is $T^*$-free, the vertex $w_1'$ is adjacent to $z_1$ and $z_1'$.
By Claim \ref{claim2}, the vertex $z_1$ has a neighbor $x$ distinct from $w_1$ and $w_1'$.
Since $G$ is $T^*$-free, the vertex $z_1'$ is adjacent to $x$.
Since $G$ is $T^*$-free, the vertex $x$ is neither $w_2$ nor $w_2'$.
By Claim \ref{claim2}, the vertex $x$ has a neighbor $y$ distinct from $z_1$ and $z_1'$.
By Claim \ref{claim1}, the vertex $w_2$ has a neighbor $z_2$ distinct from $v_2$.
Since $G$ is $T^*$-free, the vertex $w_2'$ is adjacent to $z_2$.
By Claim \ref{claim2}, the vertex $w_2$ has a neighbor $z_2'$ distinct from $v_2$ and $z_2$.
Since $G$ is $T^*$-free, the vertex $w_2'$ is also adjacent to $z_2'$,
and $z_2$ and $z_2'$ are both distinct from $y$.
Arguing as above, we obtain that
$z_2$ and $z_2'$ have a common neighbor $x'$ distinct from $w_2$ and $w_2'$,
and
$x'$ has a neighbor $y'$ distinct from $z_2$ and $z_2'$.
If $y$ equals $y'$, the graph is as in Figure \ref{fig5}, which is a contradiction. Note that $y$ cannot have degree $3$, since $G$ is $T^*$-free.

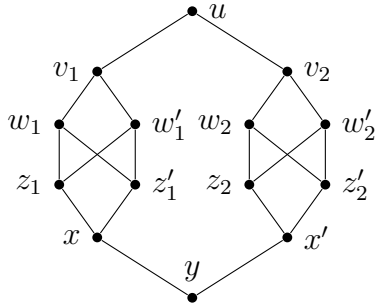
\begin{figure}[H]
\begin{center}
\begin{tikzpicture}
\node[fill, circle, inner sep=1.3pt, label=right:$u$] (u) at (0.25,3) {};
\node[fill, circle, inner sep=1.3pt, label=left:$v_1$] (v1) at (-1,2.2) {};
\node[fill, circle, inner sep=1.3pt, label=right:$v_2$] (v2) at (1.5,2.2) {};

\node[fill, circle, inner sep=1.3pt, label=left:$w_1$] (w1) at (-1.5,1.5) {};
\node[fill, circle, inner sep=1.3pt, label=right:$w_1'$] (w2) at (-0.5,1.5) {};

\node[fill, circle, inner sep=1.3pt, label=left:$z_1$] (z1) at (-1.5,0.7) {};
\node[fill, circle, inner sep=1.3pt, label=right:$z_1'$] (z2) at (-0.5,0.7) {};

\node[fill, circle, inner sep=1.3pt, label=left:$x$] (x) at (-1,0) {};

\node[fill, circle, inner sep=1.3pt, label=left:$w_2$] (w3) at (1,1.5) {};
\node[fill, circle, inner sep=1.3pt, label=right:$w_2'$] (w4) at (2,1.5) {};

\node[fill, circle, inner sep=1.3pt, label=left:$z_2$] (z3) at (1,0.7) {};
\node[fill, circle, inner sep=1.3pt, label=right:$z_2'$] (z4) at (2,0.7) {};

\node[fill, circle, inner sep=1.3pt, label=right:$x'$] (x1) at (1.5,0) {};
\node[fill, circle, inner sep=1.3pt, label=above:$y$] (y) at (0.25,-0.8) {};

\node(h1) at (-1,0) {};
\node(h2) at (1,0) {};
\node(h3) at (2,0) {};

\draw(u)--(v1);
\draw (u)--(v2);
\draw(w1)--(v1);
\draw (w2)--(v1);
\draw(w1)--(z1);
\draw(w1)--(z2);
\draw(w2)--(z1);
\draw(w2)--(z2);
\draw(z1)--(x);
\draw(z2)--(x);
\draw(v2)--(w3);

\draw(v2)--(w4);
\draw(z3)--(w3);
\draw(z4)--(w3);
\draw(z3)--(w4);
\draw(z4)--(w4);
\draw(z3)--(x1);
\draw(z4)--(x1);

\draw(y)--(x);
\draw(y)--(x1);

\end{tikzpicture}
\end{center}
\caption{The graph $G$ in the case that $y=y'$.
Note that $G$ has a maximal matching missing $v_1$ and $x$,
which implies $\gamma_e(G)\leq \frac{2n}{3}-\frac{m}{6}$.}\label{fig5}
\end{figure}
Hence, the vertices $y$ and $y'$ are distinct.

Let $G'''=G-\{ u,v_1,v_2,w_1,w_1',w_2,w_2',z_1,z_1',z_2,z_2',x,y,x',y'\}$.
Since
$n'''=n-15$,
$m'''\geq m-24$, and
adding $\{ v_1w_1,v_2w_2,w_1'z_1,w_2'z_2,xy,x'y'\}$ to a maximal matching in $G''$
yields a maximal matching in $G$,
the choice of $G$ implies
\begin{eqnarray*}
\gamma_e(G)
& \leq & \frac{2}{3}n'''-\frac{1}{6}m'''+6
\leq \frac{2}{3}(n-15)-\frac{1}{6}(m-24)+6
=\frac{2}{3}n-\frac{1}{6}m.
\end{eqnarray*}
This contradiction completes the proof of the bound.
Clearly, our proof yields an efficient recursive algorithm constructing
a maximal matching whose size is at most the right hand side of (\ref{e10}),
which completes the proof.
\end{proof}
Now, we can proceed to the proof of Theorem \ref{theorem3}.

\begin{proof}[Proof of Theorem \ref{theorem3}]
Let $uv$ be any edge of $G$.
Let $G'=G-\{ u,v\}$.
Let $n'$ and $m'$ be the order and size of $G'$, respectively.
Note that $n'=n-2$ and $m'=\frac{3n}{2}-5$.
Since adding $uv$ to a maximal matching in $G'$ yields a maximal matching in $G$,
and no component of $G'$ is cubic,
Theorem \ref{theorem1} implies
\begin{eqnarray}
\gamma_e(G) &\leq &
\frac{2}{3}(n-2)-\frac{1}{6}\left(\frac{3n}{2}-5\right)+1
=\frac{5n+6}{12}.
\end{eqnarray}
\end{proof}
Similarly as for Theorem \ref{theorem3},
the proof of Theorem \ref{theorem4} is based on the following result.

\begin{theorem}\label{theorem2}
If $G$ is a graph of order $n$, size $m$, and maximum degree at most $\Delta$
for some integer $\Delta\geq 3$
such that no component of $G$ is $\Delta$-regular, then
\begin{eqnarray}\label{e2}
\gamma_e(G)\leq \frac{\Delta n}{2\Delta-1}-\frac{m}{(\Delta-1)(2\Delta-1)},
\end{eqnarray}
with equality if and only if
\begin{itemize}
\item $\Delta$ is odd and every component of $G$ is $K_{\Delta-1,\Delta-1}$, or
\item $\Delta$ is even and every component of $G$ is $K_{\Delta-1,\Delta-1}$ or $K_{\Delta}$.
\end{itemize}
Furthermore,
a maximal matching whose size is at most the right hand side of (\ref{e2})
can be found efficiently.
\end{theorem}
For the proof of Theorem \ref{theorem2},
we need two preparatory lemmas.

\begin{lemma}\label{lemma1}
Let $G$ be a graph.
If $u$ is a vertex of $G$ of minimum degree $\delta$
that has a neighbor $v$ of degree larger than $\delta$,
then there is a matching $M$ in $G$
with $N_G(u)\subseteq V(M)$ and $u\not\in V(M)$
such that every edge in $M$ has at least one endpoint in $N_G(u)$.
Furthermore, such a matching can be found efficiently.
\end{lemma}
\begin{proof}
We will show that a matching with the desired properties can be constructed greedily
by using $M$-alternating paths with one, two, or three edges.
Suppose that we have already constructed a matching $M$ such that
\begin{enumerate}[(i)]
\item no edge in $M$ is incident with $u$ or $v$, and
\item all edges in $M$ are incident with a vertex in $N_G(u)$.
\end{enumerate}
Initially, $M$ can be chosen empty.
We now explain how to efficiently replace $M$ with a matching
that either has the desired properties
or satisfies (i) and (ii)
and covers one more vertex in $N_G(u)\setminus \{ v\}$ than $M$.

Let $M$ be $\{ x_1y_1,\ldots,x_{k+\ell}y_{k+\ell}\}$,
where the vertices $x_1,\ldots,x_{k+\ell}$ and $y_1,\ldots,y_k$ are neighbors of $u$
but the vertices $y_{k+1},\ldots,y_{k+\ell}$ are not neighbors of $u$,
that is, $k$ edges in $M$ have both their endpoints in $N_G(u)$
while $\ell$ edges in $M$ have only one endpoint in $N_G(u)$.
Let $R=N_G(u)\setminus V(M)$, and $R'=R\setminus \{ v\}$.
Since $v\in R$, we have $|R|=\delta-(2k+\ell)>0$.

If $R'$ is not independent,
then we can simply extend $M$
by adding any edge between two vertices in $R'$.
Hence, we may assume that $R'$ is independent.
If there is some edge $e$ between a vertex in $R'$ and a vertex outside of $N_G[u]\cup V(M)$,
then we can extend $M$ by adding $e$.
Hence, we may assume that there are no such edges.

First, suppose that $|R'|=\delta-(2k+\ell)-1\geq 2$.
Let $w_1$ and $w_2$ be two vertices in $R'$.
If there is an edge between $\{ w_1,w_2\}$ and $\{ x_{k+1},\ldots,x_{k+\ell}\}$, say $w_jx_i$,
then we can replace $M$ by $(M\setminus \{ x_iy_i\})\cup \{ w_jx_i\}$.
Hence, we may assume that there are at most $2\ell$ edges between $\{ w_1,w_2\}$
and $\{ x_{k+1},\ldots,x_{k+\ell},y_{k+1},\ldots,y_{k+\ell}\}$.
If there are more than two edges between $\{ w_1,w_2\}$ and $\{ x_i,y_i\}$ for some $i\in [k]$,
then there is an $M$-augmenting path of the form $w_jx_iy_iw_{3-j}$ for some $j\in [2]$,
and we can replace $M$ by $(M\setminus \{ x_iy_i\})\cup \{ w_jx_i,y_iw_{3-j}\}$.
Hence, we may assume that there are at most two edges
between $\{ w_1,w_2\}$ and $\{ x_i,y_i\}$ for every $i\in [k]$.
Now, since, $w_1$ and $w_2$ are adjacent to $u$ and possibly $v$,
and $2k+\ell\leq \delta-3$,
we obtain $d_G(w_1)+d_G(w_2)\leq 2+2+2\ell+2k\leq 4+2(\delta-3)<2\delta$.
This implies the contradiction that $\min\{ d_G(w_1),d_G(w_2)\}<\delta$,
which implies that one of the different ways of extending $M$ considered above must always be possible.

Next, suppose that $|R'|=\delta-(2k+\ell)-1=1$.
Let $R'=\{ w\}$.
If $v$ and $w$ are adjacent,
then we add $vw$ to $M$ completing the construction of the desired matching.
Hence, we may assume that $v$ and $w$ are not adjacent.
Arguing similarly as above,
we may assume that $w$ has no neighbor in $\{ x_{k+1},\ldots,x_{k+\ell}\}$,
which implies the contradiction $d_G(w)\leq 1+\ell+2k=\delta-1$.
Again, this contradiction implies that one of the different ways of extending $M$
considered above must always be possible.

Finally, suppose that $R'$ is empty.
If $v$ has a neighbor $v'$ outside of $N_G[u]\cup V(M)$,
then we add $vv'$ to $M$ completing the construction of the desired matching.
Hence, we may assume that $v$ has no such neighbor.
Arguing similarly as above,
we may assume that $v$ has no neighbor in $\{ x_{k+1},\ldots,x_{k+\ell}\}$,
which implies $d_G(v)\leq 1+\ell+2k=\delta$.
Since the degree of $v$ is larger than $\delta$,
this is a contradiction,
which again implies that one of the different ways of extending $M$
considered above must always be possible.
This completes the proof.
\end{proof}

\begin{lemma}\label{lemma2}
Let $\delta\geq 2$.
Let $G$ be a connected $\delta$-regular graph
that is distinct from $K_{\delta,\delta}$
and, if $\delta$ is odd, is also distinct from $K_{\delta+1}$.

If $u$ is any vertex of $G$,
then there is a matching $M$ in $G$
with $N_G(u)\subseteq V(M)$ and $u\not\in V(M)$
such that every edge in $M$ has at least one endpoint in $N_G(u)$.
Furthermore, such a matching can be found efficiently.
\end{lemma}
\begin{proof}
First, suppose that $N_G(u)$ is independent.
Let $B$ be the set of vertices of $G$ at distance two from $u$.
Let $H$ be the bipartite graph with partite sets $N_G(u)$ and $B$
that contains all edges of $G$ between these two sets.
Let $M$ be a maximum matching in $H$.
If $M$ covers every vertex of $N_G(u)$, it has the desired properties.
Hence, we may assume that there is some neighbor $v$ of $u$ that is not covered by $M$.
By the independence of $N_G(u)$, and the regularity of $G$,
it follows that $|M|=\delta-1$, and $N_G(v)=(V(M)\setminus N_G(u))\cup \{ u\}$.
Now, since $M$ is maximum,
it follows that $N_G(w)=(V(M)\setminus N_G(u))\cup \{ u\}$ for every vertex $w$ in $N_G(u)$,
which implies the contradiction that $G$ is $K_{\delta,\delta}$.

Next, suppose that $N_G(u)$ is not independent.
Similarly as in the proof of Lemma \ref{lemma1},
we will show that a matching with the desired properties can be constructed greedily.
Suppose that we have already constructed a matching $M$ such that
\begin{enumerate}[(i)]
\item no edge in $M$ is incident with $u$,
\item all edges in $M$ are incident with a vertex in $N_G(u)$, and
\item at least one edge in $M$ has both its endpoints in $N_G(u)$.
\end{enumerate}
Initially, $M$ can be chosen to contain exactly one edge within $N_G(u)$.
We now explain how to efficiently replace $M$ with a matching
that either has the desired properties
or satisfies (i), (ii), and (iii),
and covers one more vertex in $N_G(u)$ than $M$.
Clearly, this implies the desired result.
Let $M$ be $\{ x_1y_1,\ldots,x_{k+\ell}y_{k+\ell}\}$
such that the notational conditions from the proof of Lemma \ref{lemma1} hold.
Note that (iii) is equivalent to $k>0$.
Let $R=N_G(u)\setminus V(M)$.

If $|R|\geq 2$, then arguing as in the case $|R'|\geq 2$ of the proof of Lemma \ref{lemma1}
implies that $M$ can be replaced as described above.
It is important to note that every replacement of $M$ considered in that proof
maintains property (iii).
Hence, we may assume that $|R|=1$.
Let $R=\{ v\}$.
Clearly, $v$ has no neighbor outside of $N_G[u]\cup V(M)$
and also no neighbor in $\{ x_{k+1},\ldots,x_{k+\ell}\}$.
Since $v$ has degree $\delta$,
this implies $N_G(v)=\{ u,x_1,\ldots,x_k,y_1,\ldots,y_{k+\ell}\}$.

If $\ell>0$, and some vertex in $\{ x_{k+1},\ldots,x_{k+\ell}\}$,
say $x_i$, has a neighbor in $\{ x_1,\ldots,x_k\}\cup \{ y_1,\ldots,y_k\}$,
say $x_j$,
then replacing $M$ with
$(M\setminus \{ x_jy_j\})\cup \{ x_ix_j,y_jv\}$ has the desired properties.
Hence, either $\ell\geq 1$ and no vertex in $\{ x_{k+1},\ldots,x_{k+\ell}\}$
has a neighbor in $\{ x_1,\ldots,x_k\}\cup \{ y_1,\ldots,y_k\}$,
or $\ell=0$.
In the first case, the degree of $x_{k+1}$ is necessarily less than $\delta$,
which is a contradiction.
Hence, we may assume that $\ell=0$,
which implies that $\delta$ is odd.
If any vertex in $N_G(u)$, say $x_i$, has a neighbor $y$ outside of $N_G[u]$,
then $(M\setminus \{ x_iy_i\})\cup \{ x_iy,vy_i\}$ has the desired properties.
Since $G$ is regular, this immediately implies the contradiction that $G$ is complete,
which completes the proof.
\end{proof}
Now, we can proceed to the proof of Theorem \ref{theorem2}.

\begin{proof}[Proof of Theorem \ref{theorem2}]
We first prove the statement about the bound and the extremal graphs,
and then argue that our proof yields an efficient algorithm to determine the desired maximal matching.
Suppose, for a contradiction, that $G$ is a counterexample of minimum order,
that is,
either (\ref{e2}) does not hold,
or (\ref{e2}) holds with equality
but $G$ is not one of the stated extremal graphs.
Clearly, this implies that $G$ is connected.
We consider two cases.

\bigskip

\noindent {\bf Case 1} {\it No vertex of $G$ has degree $\Delta$.}

\bigskip

\noindent
If no vertex of $G$ has degree $\Delta$,
then $m\leq \frac{(\Delta-1)n}{2}$, and we obtain
\begin{eqnarray}\label{e3}
\gamma_e(G)\leq \frac{n}{2}
=\frac{\Delta n}{2\Delta-1}-\frac{(\Delta-1)n}{2(\Delta-1)(2\Delta-1)}
\leq \frac{\Delta n}{2\Delta-1}-\frac{m}{(\Delta-1)(2\Delta-1)},
\end{eqnarray}
that is, (\ref{e2}) holds.
By the choice of $G$,
it follows that (\ref{e2}) holds with equality
but $G$ is not one of the stated extremal graphs.
Equality in the first inequality in (\ref{e3}) implies that every maximal matching in $G$ is perfect,
and equality in the second inequality in (\ref{e3}) implies that $G$ is $(\Delta-1)$-regular.
Since $G$ is not one of the stated extremal graphs,
Lemma \ref{lemma2} implies the existence of a matching $M$ in $G$
with $N_G(u)\subseteq V(M)$ and $u\not\in V(M)$ for some vertex $u$ of $G$.
Since no maximal matching containing $M$ covers $u$,
we obtain a contradiction,
which completes the proof in this case.

\bigskip

\noindent {\bf Case 2} {\it Some vertex of $G$ has degree $\Delta$.}

\bigskip

\noindent Let $\delta$ be the minimum degree of $G$.
Considering a path between a vertex of degree $\delta$ and a vertex of degree $\Delta$,
we can efficiently find a vertex $u$ of minimum degree $\delta$
that has a neighbor $v$ of degree larger than $\delta$.
By Lemma \ref{lemma1},
we can efficiently construct a matching $M$ in $G$
with the properties stated in Lemma \ref{lemma1}.
Let $M$ contain exactly $k$ edges with both endpoints in $N_G(u)$,
that is, $|M|=k+(\delta-2k)=\delta-k$.
Let $G'=G-(V(M)\cup \{ u\})$.
Let $n'$ and $m'$ be the order and size of $G'$, respectively.
Note that $n'=n-(2|M|+1)=n-2\delta+2k-1$.
Let $m''=m-m'$.
Since $u$ has degree $\delta$,
and the subgraph of $G$ induced by $V(M)\cup \{ u\}$ has at least $|M|+\delta$ edges,
we obtain $m''\leq (\delta+2\Delta|M|)-(|M|+\delta)=(2\Delta-1)|M|=(2\Delta-1)(\delta-k)$,
which implies
\begin{eqnarray}\label{e4}
m'\geq m-(2\Delta-1)(\delta-k).
\end{eqnarray}
Note that equality holds in (\ref{e4}) if and only if
every vertex in $V(M)$ has degree $\Delta$,
and the subgraph of $G$ induced by $V(M)\cup \{ u\}$ has exactly $|M|+\delta$ edges.

Since adding $M$ to a maximal matching in $G'$ yields a maximal matching in $G$,
no component of $G'$ is $\Delta$-regular, and
$\Delta-1-\delta+k\geq 0$,
the choice of $G$ implies
\begin{eqnarray}
\gamma_e(G) &\leq & \gamma_e(G')+|M| \nonumber \\
& \leq & \frac{\Delta n'}{2\Delta-1}-\frac{m'}{(\Delta-1)(2\Delta-1)}+|M| \label{e5}\\
& \leq & \frac{\Delta (n-2\delta+2k-1)}{2\Delta-1}-\frac{m-(2\Delta-1)(\delta-k)}{(\Delta-1)(2\Delta-1)}+\delta-k \label{e6}\\
& = & \frac{\Delta n}{2\Delta-1}-\frac{m}{(\Delta-1)(2\Delta-1)}
-\frac{\Delta(\Delta-1-\delta+k)}{(\Delta-1)(2\Delta-1)}\nonumber \\
& \leq & \frac{\Delta n}{2\Delta-1}-\frac{m}{(\Delta-1)(2\Delta-1)}, \label{e7}
\end{eqnarray}
that is, (\ref{e2}) holds.
By the choice of $G$,
it follows that (\ref{e2}) holds with equality
but $G$ is not one of the stated extremal graphs.
This implies that equality holds in (\ref{e5}), (\ref{e6}), and (\ref{e7}).
Equality in (\ref{e5}) implies that every component of $G'$ is one of the stated extremal graphs.
Equality in (\ref{e7}) implies that $k=0$, and that $\delta=\Delta-1$,
which implies $|M|=\Delta-1$.
Equality in (\ref{e6}) implies equality in (\ref{e4}), which implies
that every vertex in $V(M)$ has degree $\Delta$,
and the subgraph of $G$ induced by $V(M)\cup \{ u\}$ has exactly $|M|+\delta$ edges.
Note that there are exactly $(2\Delta-3)(\Delta-1)$ edges between $V(M)$ and $V(G')$.

If $\Delta$ is even, then, since $u$ is not the only vertex of odd minimum degree $\Delta-1$,
some vertex $u'$ in $V(G')$ has degree $\Delta-1$ in $G$.
Since $G$ is connected, it is easy to see that $u'$ can be chosen in such a way that
it has a neighbor $v'$ of degree $\Delta$.
Let $H$ be the component of $G'$ that contains $u'$.
Since $H$ is either $K_{\Delta}$ or $K_{\Delta-1,\Delta-1}$,
there is a matching $M'$ in $G$ and a vertex $y$ in $V(M)$
with $V(M')=(V(H)\setminus \{ u'\})\cup \{ y\}$.
By symmetry between $(u,v,M)$ and $(u',v',M')$,
it follows that
$N_G(u')$ is independent,
and that the subgraph of $G$ induced by $V(M')\cup \{ u'\}$ has exactly $|M'|+\delta$ edges.
The first property implies that $H$ is not $K_{\Delta}$,
and the second property together with $\Delta\geq 3$ implies that $H$ is not $K_{\Delta-1,\Delta-1}$.
Hence, $\Delta$ is odd, which implies that every component of $G'$ is $K_{\Delta-1,\Delta-1}$.

If there is a vertex of degree $\Delta-1$ in $G$ that is distinct from $u$,
then arguing similarly as for the vertex $u'$ considered above yields a contradiction.
Hence, $u$ is the only vertex of degree $\Delta-1$ in $G$.
This implies that there are exactly $2(\Delta-1)$ edges between every component of $G'$ and $V(M)$.
Hence, if $p$ denotes the number of components of $G'$,
then $p=\frac{(2\Delta-3)(\Delta-1)}{2(\Delta-1)}=\frac{2\Delta-3}{2}$.
Since this is never integral, we obtain a contradiction,
which completes the proof in this case.

\bigskip

\noindent Since $u$, $v$, and $M$ considered above can be found efficiently,
our proof yields an efficient recursive algorithm constructing
a maximal matching whose size is at most the right hand side of (\ref{e2}).
\end{proof}
We are now in a position to prove Theorem \ref{theorem4}.

\begin{proof}[Proof of Theorem \ref{theorem4}]
Let $uv$ be any edge of $G$.
Let $G'=G-\{ u,v\}$.
Let $n'$ and $m'$ be the order and size of $G'$, respectively.
Note that $n'=n-2$ and $m'=\frac{\Delta n}{2}-(2\Delta-1)$.
Since adding $uv$ to a maximal matching in $G'$ yields a maximal matching in $G$, and
no component of $G'$ is $\Delta$-regular,
Theorem \ref{theorem2} implies
\begin{eqnarray}
\gamma_e(G) &\leq & \gamma_e(G')+1\nonumber\\
& \leq & \frac{\Delta n'}{2\Delta-1}-\frac{m'}{(\Delta-1)(2\Delta-1)}+1\label{e8}\\
& = & \frac{\Delta (n-2)}{2\Delta-1}-\frac{\frac{\Delta n}{2}-(2\Delta-1)}{(\Delta-1)(2\Delta-1)}+1\nonumber \\
& = & \frac{\Delta(2\Delta-3)n+2\Delta}{2(\Delta-1)(2\Delta-1)},\nonumber
\end{eqnarray}
which implies (\ref{e1}).
Equality in (\ref{e1}) implies equality in (\ref{e8}),
which, by Theorem \ref{theorem2},
implies that every component of $G'$ is either $K_{\Delta-1,\Delta-1}$
or $K_{\Delta}$, where $\Delta$ has to be even in the latter case.
Since $G$ is $\Delta$-regular, and there are exactly $2\Delta-2$ edges between $\{ u,v\}$ and $V(G')$,
it follows that $G'$ is $K_{\Delta-1,\Delta-1}$.
Since $uv$ was an arbitrary edge of $G$,
it follows, by symmetry, that $G$ is $K_{\Delta,\Delta}$,
which completes the proof of (\ref{e3}) and the statement about the extremal graph.
\end{proof}

For the proof of Theorem \ref{theorem5},
we need lower bounds on the edge domination number.
\begin{lemma}\label{lemma3}
Let $G$ be a $\Delta$-regular graph of order $n$ for some $\Delta\geq 3$.
\begin{enumerate}[(i)]
\item $\gamma_e(G)\geq \frac{\Delta n}{4\Delta-2}$.
\item If $G$ is claw-free, then $\gamma_e(G)\geq \frac{\Delta n}{2\Delta+4}$.
\end{enumerate}
\end{lemma}
\begin{proof}
Let $M$ be a minimum maximal matching in $G$.
Let $m$ be the number of edges between $V(M)$ and $V(G)\setminus V(M)$.
Since $V(G)\setminus V(M)$ is an independent of order $n-2\gamma_e(G)$,
it follows that $m$ is exactly $\Delta(n-2\gamma_e(G))$.
Since every edge $uv$ in $M$ contributes at most $d_G(u)+d_G(v)-2$ edges to $m$,
it follows that $m$ is at most $(2\Delta-2)\gamma_e(G)$.
Now, (i) follows immediately from $\Delta(n-2\gamma_e(G))=m\leq (2\Delta-2)\gamma_e(G)$.
If $G$ is claw-free, then every vertex in $V(M)$ has at most two neighbors in $V(G)\setminus V(M)$.
This implies $\Delta(n-2\gamma_e(G))=m\leq 4\gamma_e(G)$, and (ii) follows.
\end{proof}
It is easy to see that equality holds in Lemma \ref{lemma3}(i)
if and only if $G$ has a dominating induced matching.
Since deciding the existence of a dominating induced matching
in a given $k$-regular graphs is NP-complete for every fixed $k\geq 3$ \cite{cacedesi},
there is little hope of characterizing the extremal graphs for Lemma \ref{lemma3}(i).
Lemma \ref{lemma3}(ii) holds with equality for the graph
that arises by removing a perfect matching from a complete graph of even order.
Actually, we believe that this is the only connected extremal graph for Lemma \ref{lemma3}(ii).

\begin{proof}[Proof of Theorem \ref{theorem5}]
Note that
\begin{eqnarray*}
\lim\limits_{n\to\infty}\left(\frac{5n+6}{12}\right)/
\left(\frac{3n}{10}\right)&=&\frac{25}{18},\\[3mm]
\lim\limits_{n\to\infty}\left(\frac{\Delta(2\Delta-3)n+2\Delta}{2(\Delta-1)(2\Delta-1)}\right)/
\left(\frac{\Delta n}{4\Delta-2}\right)&=&2-\frac{1}{\Delta-1},\mbox{ and }\\[3mm]
\lim\limits_{n\to\infty}\left(\frac{\Delta(2\Delta-3)n+2\Delta}{2(\Delta-1)(2\Delta-1)}\right)/
\left(\frac{\Delta n}{2\Delta+4}\right)&=& 1+\frac{4\Delta-7}{2\Delta^2-3\Delta+1}.
\end{eqnarray*}
Now, the statements follow easily
from the algorithmic statements in Theorems \ref{theorem1} and \ref{theorem2},
and the proofs of Theorems \ref{theorem3} and \ref{theorem4}.
More precisely,
after removing any two adjacent vertices as in the proofs of Theorems \ref{theorem3} and \ref{theorem4},
one solves the minimum maximal matching problem exactly
on all components up to a sufficiently large order depending on $\epsilon$,
and applies the algorithmic statements from Theorems \ref{theorem1} and \ref{theorem2}
to the larger components.
\end{proof}

\end{document}